\documentclass[leqno]{amsart}
\usepackage{amsmath,amsfonts,amsthm,amssymb,indentfirst,epic,url}

\newtheorem{theorem}{Theorem}
\newtheorem{lemma}[theorem]{Lemma}
\newtheorem{corollary}[theorem]{Corollary}

\newtheorem{question}[theorem]{Question}

\newcommand{\Q}{\mathbb Q}
\newcommand{\Z}{\mathbb Z}
\renewcommand{\r}{\mathrm}

\begin{document}

\begin{center}
\texttt{Comments, corrections, and related references welcomed,
as always!}\\[.5em]
{\TeX}ed \today
\vspace{2em}
\end{center}

\title%
{Bilinear maps on Artinian modules}
\thanks{This preprint is readable online at
\url{http://math.berkeley.edu/~gbergman/papers/}.
}


\author{George M. Bergman}
\address{University of California\\
Berkeley, CA 94720-3840, USA}
\email{gbergman@math.berkeley.edu}

\begin{abstract}
It is shown that if a bilinear map $f:A\times B\to C$ of modules
over a commutative ring $k$ is nondegenerate (i.e.,
if no nonzero element of $A$ annihilates all of $B,$ and vice
versa), and $A$ and $B$ are Artinian, then
$A$ and $B$ are of finite length.

Some consequences are noted.
Counterexamples are given to some attempts to generalize the above
statement to balanced bilinear maps
of bimodules over noncommutative rings, while the
question is raised whether other such generalizations are true.
\end{abstract}
\maketitle

Rings and algebras will be understood to be associative and
unital, except where the contrary is stated.

Examples of modules over a commutative
ring $k$ that are Artinian but not Noetherian are
well known; for example, the $\!\Z\!$-module $\Z_{p^\infty}.$
However, such modules do not show up as underlying modules of
$\!k\!$-algebras.
(We shall see in \S\ref{S.HL} that this can be deduced, though
not trivially, from the Hopkins-Levitzki Theorem, which says
that left Artinian rings are also left Noetherian.)
The result of the next section can be thought of as a generalization
of this fact.

I am grateful to J.\,Krempa for pointing out two misstatements
in the first version of this note, and to D.\,Herbera, K.\,Goodearl,
T.\,Y.\,Lam and L.\,W.\,Small for references to related literature,

\section{Our main result}\label{S.main}

In the proof of the following theorem, it is
striking that everything before the next-to-last sentence works,
mutatis mutandis, if $A$ and $B$ are both
assumed Noetherian rather than Artinian, though the
final conclusion is clearly false in that case.
(On the other hand, the argument does not work at all if
one of $A$ and $B$ is assumed Artinian, and the other Noetherian.)

\begin{theorem}\label{T.main}
Suppose $k$ is a commutative ring, and
$f:A\times B\to C$ a bilinear map of $\!k\!$-modules
which is nondegenerate, in the sense that for every nonzero
$a\in A,$ the induced map $f(a,-):B\to C$ is nonzero, and for
every nonzero $b\in B,$ the induced map $f(-,b):A\to C$ is nonzero.

Then if $A$ and $B$ are Artinian, they both have finite length.
\end{theorem}

\begin{proof}
If elements $a\in A,\ b\in B$ satisfy $f(a,b)=0,$ we shall say
they \emph{annihilate} one another.
(The concept of an element $c\in k$ annihilating an element
$x$ of $A,$  $B,$ or $C$ will retain its usual meaning, $c\,x=0.)$
For subsets $Y\subseteq B,$ respectively $X\subseteq A,$ we
define the annihilator sets
\begin{equation}\begin{minipage}[c]{35pc}\label{d.perp}
$Y^\perp\ =\ \{a\in A\mid(\forall\,y\in Y)\ f(a,y)=0\}
\ \subseteq\ A,$\\[0.3em]
$X^\perp\ =\ \{b\in B\mid(\forall\,x\in X)\ f(x,b)=0\}
\ \subseteq\ B.$
\end{minipage}\end{equation}

We see that these are submodules of $A$ and $B$ respectively,
that the set of annihilator submodules in $A$ (respectively, in $B)$
forms a lattice (in the order-theoretic sense) under inclusion,
and that these two lattices of annihilator submodules
are antiisomorphic to one another, via the maps $U\mapsto U^\perp.$
(This situation is an example of a ``Galois connection''
\cite[\S5.5]{245}, but I will not assume familiarity with
that formalism.)

When $A$ and $B$ are Artinian, these lattices of submodules
of $A$ and $B$ both have descending chain condition; so since
they are antiisomorphic, they also have ascending chain condition.
Hence all their chains have finite length.
Let us choose a maximal
(i.e., unrefinable) chain of annihilator submodules of $A,$
\begin{equation}\begin{minipage}[c]{35pc}\label{d.A_i}
$\{0\}=A_0\ \subseteq\ A_1\ \subseteq\ \dots\ \subseteq A_n=A.$
\end{minipage}\end{equation}
This yields a maximal chain of annihilator submodules of $B,$
\begin{equation}\begin{minipage}[c]{35pc}\label{d.B_i}
$B=B_0\ \supseteq\ B_1\ \supseteq\ \dots\ \supseteq B_n=\{0\},$
\end{minipage}\end{equation}
where
\begin{equation}\begin{minipage}[c]{35pc}\label{d.*}
$B_i\ =\ A_i^\perp,\qquad A_i\ =\ B_i^\perp.$
\end{minipage}\end{equation}

It is easy to see that for each $i,$ $f$ induces a $\!k\!$-bilinear map
\begin{equation}\begin{minipage}[c]{35pc}\label{d.f_i}
$f_i:\ (A_{i+1}/A_i)\ \times\ (B_i/B_{i+1})\ \to\ C,$
\end{minipage}\end{equation}
via
\begin{equation}\begin{minipage}[c]{35pc}\label{d.f_i_def}
$f_i(a+A_i,\,b+B_{i+1})\ =\ f(a,b)$\quad $(a\in A_{i+1},\ b\in B_i).$
\end{minipage}\end{equation}

I claim that under $f_i,$ every nonzero element of $A_{i+1}/A_i$ has
zero annihilator in $B_i/B_{i+1},$ and vice versa.
For if we had any counterexample, say $a\in A_{i+1}/A_i,$ then
its annihilator would be a proper nonzero annihilator submodule of
$B_i/B_{i+1},$ and this would lead to an annihilator submodule of $B$
strictly between $B_i$ and $B_{i+1},$ contradicting the maximality
of the chain~(\ref{d.B_i}).

From this we can deduce that every nonzero element of $A_{i+1}/A_i$
and every nonzero element of $B_i/B_{i+1}$ have the same
annihilator in $k.$
Indeed, if $c\in k$ annihilates the nonzero element
$x\in A_{i+1}/A_i,$ then for every $y\in B_i/B_{i+1},$
$cy$ will annihilate $x,$ hence must be zero; so
every $c\in k$ annihilating one nonzero member of $A$
annihilates all of $B,$ and dually.

It follows that the common annihilator of all nonzero elements
of these two modules is a prime ideal $P_i\subseteq k,$ making
$k/P_i$ an integral domain, such that $A_{i+1}/A_i$
and $B_i/B_{i+1}$ are $\!k/P_i\!$-modules.
Moreover, taking any nonzero element of either of our modules,
say $x\in A_{i+1}/A_i,$ we have $kx\cong k/P_i$ as modules, so
since $A_{i+1}/A_i$ is Artinian, so is $k/P_i.$

But an Artinian integral domain is a field; so $A_{i+1}/A_i$
and $B_i/B_{i+1}$ are vector spaces over the field $k/P_i,$
so the fact that they are Artinian means that they have finite length.
Thus, $A$ and $B$ admit finite chains~(\ref{d.A_i}),~(\ref{d.B_i})
of submodules with factor-modules of finite length; so
they are each of finite length.
\end{proof}

\section{Some immediate consequences}\label{S.cor}

We start with a trivial consequence of our theorem.

\begin{corollary}\label{C.nonnondeg}
If in the statement of Theorem~\ref{T.main}, we assume only one of
the nondegeneracy conditions, namely that for all nonzero
$a\in A,$ the induced map $f(a,-):B\to C$ is nonzero
\textup{(}respectively, that for all nonzero $b\in B,$
the induced map $f(-,b):A\to C$ is nonzero\textup{)},
we can still conclude that $A$ \textup{(}respectively,
$B)$ has finite length.

Assuming neither nondegeneracy condition, we can still
conclude that $A/B^\perp$ and $B/A^\perp$ have finite length.
\end{corollary}

\begin{proof}
Without any nondegeneracy assumption, note that $f$ induces
a \emph{nondegenerate} bilinear map
\begin{equation}\begin{minipage}[c]{35pc}\label{d./perp}
$A/B^\perp\ \times\ B/A^\perp\ \to\ C.$
\end{minipage}\end{equation}
Since $A/B^\perp$ and $B/A^\perp$ are again Artinian, we can apply
Theorem~\ref{T.main} to~(\ref{d./perp}) and conclude that
both these factor-modules have finite length.

On the other hand, the two nondegeneracy conditions in the first
assertion of the corollary are, respectively,
equivalent to $A=A/B^\perp$ and to $B=B/A^\perp;$ so
combining one or the other of these
with the above result, we get the first asserted conclusion.
\end{proof}

We can now recover a result of A.\,Facchini, C.\,Faith and D.\,Herbera.

\begin{corollary}[{\cite[Proposition~6.1]{CF+DH}}]\label{C.oplus}
The tensor product $A\otimes_k B$ of two Artinian modules over a
commutative ring $k$ has finite length.
\end{corollary}

\begin{proof}
Letting $A^\perp$ and $B^\perp$ be annihilators
with respect to the tensor multiplication
$\otimes:A\times B\to A\otimes_k B,$ the preceding corollary
tells us that $A/B^\perp$ and $B/A^\perp$ have finite length.
But $A\otimes_k B$ can also be regarded as the tensor product of
these factor modules; and a tensor product of modules of finite
length over a commutative ring has finite length.
\end{proof}

Let us next apply Theorem~\ref{T.main}
to the multiplication of a $\!k\!$-algebra $R.$
This does not require $R$
to be associative, so we shall make no such assumption.
In the study of nonassociative algebras, it is
often not natural to require a unit; but without one,
nondegeneracy of the multiplication is not automatic; so
in the statement below we get this nondegeneracy by
dividing out by an appropriate annihilator ideal.

(Caveat: below, we name that ideal $Z(R),$
though it is not the center of $R.$
This notation, from \cite{prod_Lie1}, is based on the
phrase ``zero multiplication'', and also on the use of that symbol in
the theory of Lie algebras, where it does coincide with the center.)

\begin{corollary}\label{C.k-algs}
Let $R$ be a not-necessarily-associative algebra over a
commutative ring $k,$ and let
\begin{equation}\begin{minipage}[c]{35pc}\label{d.Z}
$Z(R)\ =\ \{x\in R\mid xR=Rx=\{0\}\,\}.$
\end{minipage}\end{equation}
Then if $R/Z(R)$ is Artinian as a $\!k\!$-module, it
is also Noetherian as a $\!k\!$-module.

Hence, if $Z(R)=\{0\}$ \textup{(}in particular, if $R$ has
a unit element\textup{)}, then if $R$ is Artinian
as a $\!k\!$-module, it is also Noetherian as a $\!k\!$-module.
\end{corollary}

\begin{proof}
Let
\begin{equation}\begin{minipage}[c]{35pc}\label{d.Z_lr}
$Z_l(R)\ =\ \{x\in R\mid xR=\{0\}\,\},$\quad
$Z_r(R)\ =\ \{x\in R\mid Rx=\{0\}\,\}.$
\end{minipage}\end{equation}
Thus, $Z(R)=Z_l(R)\cap Z_r(R).$
The multiplication of $R$ induces a nondegenerate
bilinear map of $\!k\!$-modules $R/Z_l(R)\times R/Z_r(R)\to R.$
Since $R/Z_l(R)$ and $R/Z_r(R)$ are homomorphic
images of $R/Z(R),$ they are Artinian over $k,$ hence
by Theorem~\ref{T.main} they are Noetherian over $k.$
Hence $R/Z(R),$ which embeds in $R/Z_l(R)\times R/Z_r(R),$
is also Noetherian.

The assertion of the final sentence clearly follows.
\end{proof}

This shows, as mentioned at the start of this note, that
groups such as $\Z_{p^\infty}$ cannot be the additive groups of rings.
There are other groups, such as $\Q/\Z,$ which one feels
should be excluded for similar reasons, though they are not
themselves Artinian.
This leads us to formulate
the following consequence of Theorem~\ref{T.main}.

\begin{corollary}\label{C.local}
Suppose $k$ is a commutative ring, and
$f:A\times B\to C$ a nondegenerate bilinear map of $\!k\!$-modules.

Then if $B$ \textup{(}respectively $A)$
is locally Artinian \textup{(}i.e.,
if every finitely generated submodule thereof is Artinian\textup{)},
then every Artinian submodule of $A$ \textup{(}respectively
$B)$ has finite length.
\end{corollary}

\begin{proof}
Let $A_0$ be an Artinian submodule of $A,$ and consider the
annihilators in $A_0$ of all finitely generated submodules of $B.$
This family of
annihilators is clearly closed under finite intersections; so by
descending chain condition on submodules of $A_0,$ the intersection
of this family is itself such an annihilator, say of
the finitely generated submodule $B_0\subseteq B.$
But by nondegeneracy of $f,$ that intersection is $\{0\};$
so $B_0$ has trivial annihilator in $A_0.$
Now assuming $B$ locally Artinian, $B_0$ will be Artinian,
so we can apply Corollary~\ref{C.nonnondeg}
to the restricted map $f:A_0\times B_0\to C,$
and conclude that $A_0$ has finite length.
By symmetry, we also have the corresponding implication
with the roles of $A$ and $B$ interchanged.
\end{proof}

One has obvious analogs of Corollaries~\ref{C.nonnondeg},
\ref{C.oplus} and~\ref{C.k-algs} for this result.
From the last of these,
we see that the $\!\Z\!$-module $\Q/\Z,$ which is locally
Artinian and has Artinian submodules of infinite length, cannot
be the additive group of a unital (associative or nonassociative) ring.

\section{Relation to the Hopkins-Levitzki Theorem}\label{S.HL}

Going back to Corollary~\ref{C.k-algs}, can the case of that result
where $R$ is an associative unital $\!k\!$-algebra be deduced from the
well-known fact that a left Artinian ring is also left Noetherian --
the Hopkins-Levitzki Theorem \cite[Theorem~4.15(A)]{TYL1}\,?

Well, if, as assumed in that
corollary, $R$ is Artinian as a $\!k\!$-module,
then it is certainly Artinian as a left $\!R\!$-module, hence
if it is associative and unital, the Hopkins-Levitzki
Theorem says it is Noetherian as a left $\!R\!$-module.
Can we somehow get from this that it is Noetherian as a $\!k\!$-module?

We can, using the following striking result.
\begin{equation}\begin{minipage}[c]{35pc}\label{d.THL+DH}
(Theorem of Lenagan and Herbera \cite[Theorem on p.2044]{GM}.)
If $R$ and $S$ are rings, and $_R M_S$ a bimodule which
is left Noetherian and right Artinian,
then $M$ is also left Artinian and right Noetherian.
\end{minipage}\end{equation}

Indeed, from~(\ref{d.THL+DH}) we deduce

\begin{corollary}[to~(\ref{d.THL+DH})]\label{C.THL+DH-}
If $R$ is a $\!k\!$-algebra, and $_R M$ an $\!R\!$-module
of finite length \textup{(}as an $\!R\!$-module\textup{)},
and if, moreover, $M$ is Artinian or Noetherian
as a $\!k\!$-module, then it has finite length as a $\!k\!$-module.
\end{corollary}

\begin{proof}
Regard $M$ as a bimodule $_R M_k.$
Then~(\ref{d.THL+DH}) or its left-right dual, applied to this bimodule,
gives the desired conclusion.
\end{proof}

So in the associative unital case of
Corollary~\ref{C.k-algs}, once we know by
the Hopkins-Levitzki Theorem that $_R R$ has finite length
as an $\!R\!$-module, the above result gives an alternative
proof of the conclusion of that corollary.

(Notes on the background of~(\ref{d.THL+DH}):
In \cite{GM},~(\ref{d.THL+DH}) is described as a result
of T.\,Lenagan, with a new proof communicated
to the author by D.\,Herbera.
Lenagan \cite{THL} had indeed proved the hard part
of~(\ref{d.THL+DH}), that if $M$ is both Artinian and Noetherian
on one side, and Noetherian on the other,
then it is also Artinian on the latter side;
and this is what is called Lenagan's Theorem in
most sources, e.g, \cite{TYL}.
However, Herbera's version in~\cite{GM}
tacitly supplies the additional argument showing
that if $M$ is merely assumed Artinian on one side and Noetherian
on the other, then it is also Noetherian on the former side.
It is that part, and not the part proved by
Lenagan, that we needed for our alternative proof
of the associative unital case of Corollary~\ref{C.k-algs}.
Incidentally, Lenagan formulated his result
for $\!2\!$-sided ideals, but as noted in
\cite[p.\,332, sentence after Theorem~11.30]{TYL},
his proof carries over verbatim to bimodules.)

For some results on when not-necessarily-unital associative
right Artinian rings must be right Noetherian, and related questions,
see \cite{Huynh} and papers cited there.

Returning to the relation between Corollary~\ref{C.k-algs} and
the Hopkins-Levitzki Theorem, we cannot hope to go the other way,
and obtain the latter from the former: We can't
get started, since the Artinian assumption
on $R$ as a left $\!R\!$-module does not, in general, by
itself give such a condition on $R$ as a $\!k\!$-module.
This suggests that we look for some result that can be applied
directly to the right and left $\!R\!$-module structures of $R;$ say a
generalization of Theorem~\ref{T.main} to a result
on balanced bilinear maps
\begin{equation}\begin{minipage}[c]{35pc}\label{d.bimods}
$f:{_S A_R} \times{_R B_T}\ \to\ {_S C_T}$
\end{minipage}\end{equation}
of bimodules over associative rings.

I have not been able to find such a generalization.
Let us take a brief look at the situation.

\section{Some counterexamples, and a question}\label{S.ceg}

Given a map~(\ref{d.bimods}), the
annihilator of every element of $A$ is a $\!T\!$-submodule
of $B,$ and the annihilator of every element of $B$ is an
$\!S\!$-submodule of $A;$ so one might hope for a result
assuming the Artinian property for $_S A$ and $B_T.$
But this does not work:
if we take a field $k,$ two $\!k\!$-algebras $S$ and $T,$
any nonzero Artinian left $\!S\!$-module $A,$
and an Artinian but non-Noetherian right $\!T\!$-module $B$
(for instance, $S=T=k[t],$ $A=B=k((t))/k[[t]],$
equivalently, $k(t)/k[t]_{(t)}),$ then we see that the canonical map
\begin{equation}\begin{minipage}[c]{35pc}\label{d.otimes}
$\otimes:\ {_S A_k} \times {_k B_T}\ \to\ {_S(A} \otimes_k B)_T$
\end{minipage}\end{equation}
is a counterexample: it is nondegenerate, and
$_S A$ and $B_T$ are Artinian, but they are not both Noetherian.

If, instead, we assume the Artinian condition on $A_R$ and $_R B,$
counterexamples are harder to find; but we can
get them using the following construction.

\begin{lemma}\label{L.simple_arb}
Let $k$ be a field, $R_0$ a $\!k\!$-algebra, and $_{R_0} M$ any
nonzero left $\!R_0\!$-module.
Then there exists a $\!k\!$-algebra $R$ having\\[.5em]
\textup{(i)} a left module $_R B$ whose lattice
of submodules is obtained, up to
isomorphism, from the lattice of submodules of $_{R_0} M$ by
adjoining one new element at the bottom,\\[.5em]
\textup{(ii)} a \emph{simple} right module $A_R,$ and\\[.5em]
\textup{(iii)} a nondegenerate $\!R\!$-balanced $\!k\!$-bilinear map
$f: A_R \times {_R B}\to k.$
\end{lemma}

\begin{proof}
From $_{R_0}M,$ we shall construct $B$ as a $\!k\!$-vector-space.
We shall then define $R$ as a certain $\!k\!$-algebra of linear
endomorphisms of $B,$ and $A$ as a certain $\!k\!$-vector space
of linear functionals $B\to k,$
closed under right composition with the actions of members of $R.$
The map $f: A\times B\to C$ will be
the function that evaluates members of $A$ at members of $B.$
Here are the details.

Writing $\omega$ for the set of natural numbers,
let $B\subseteq M^\omega$ be the vector space of those sequences
$x = (x_i)_{i\in\omega}$ which have the same
value at all but finitely many $i.$
Let $R$ be the $\!k\!$-algebra of $\!k\!$-linear maps $B\to B$
spanned by
\begin{equation}\begin{minipage}[c]{35pc}\label{d.R}
those maps which act by multiplication by an element
$r\in R_0$ simultaneously on all coordinates; i.e., by
$r x = (r x_i)_{i\in\omega},$
\end{minipage}\end{equation}
and
\begin{equation}\begin{minipage}[c]{35pc}\label{d.fin}
those maps which act by projecting to the sum of finitely
many of our copies of $M,$ then mapping this sum into itself
by an arbitrary finite-rank $\!k\!$-vector-space endomorphism.
\end{minipage}\end{equation}

It is easy to see that every nonzero $\!R\!$-submodule of $B$ contains
the submodule $B_\r{fin}$ of all elements having finite support
in $\omega,$ that an $\!R\!$-submodule $B'$ containing $B_\r{fin}$ is
determined by the values
that elements of $B'$ assume ``almost everywhere'', and that the set
of these values can be, precisely, any submodule of $M.$
Thus, the lattice of
submodules of $B$ containing $B_\r{fin}$ is isomorphic to the lattice
of submodules of $M;$ so the full lattice of submodules of $B$
consists of this and a new bottom element, the zero submodule.

Let $A$ be the set of all $\!k\!$-linear functionals $a$ on $B$ that
depend on only finitely many coordinates (i.e., for which there
exists a finite subset $I\subseteq\omega$ such that $a$ factors
through the projection of $B\subseteq M^\omega$ to $M^I).$
This set is easily seen to be closed
under right composition with elements of $R;$ hence we may regard
$A$ as a right $\!R\!$-module, and evaluation of
elements of $A$ on elements of $B$ gives an $\!R\!$-balanced
$\!k\!$-bilinear map $f: A\times B\to k.$
Further, for any $a\in A-\{0\}$ and $b\in B-\{0\},$
we can clearly find a $u\in R$ of the sort described in~(\ref{d.fin})
which carries $b$ to an element not in $\r{ker}(a),$ so that
$0\neq f(a,ub) = f(au,b).$
In particular, $f$ has the two properties defining nondegeneracy
(statement of Theorem~\ref{T.main}).

It remains to show that  $A_R$ is simple.
Given $a\in A-\{0\}$ and $a'\in A,$ we see that  there will exist
$y\in B$ with finite
support such that $a(y)=1.$
Choosing such a $y,$ define $u: B\to B$ by $u(x)=a'(x) y.$
It is easy to check that $u$ has the form~(\ref{d.fin}),
hence lies in $R,$ and that it satisfies $au = a',$ proving simplicity.
\end{proof}

Taking for $_{R_0} M$ in the above
lemma any Artinian non-Noetherian module over
a $\!k\!$-algebra $R_0,$ we get, as desired,
a nondegenerate balanced bilinear map~(\ref{d.bimods})
with $A_R$ and $_R B$ Artinian
(and $A_R$ Noetherian), but with $_R B$ non-Noetherian.

However, neither the examples obtained using~(\ref{d.otimes})
nor those gotten as above
satisfy all four possible Artinian conditions on $A$ and $B.$
So we ask

\begin{question}\label{Q.4artin}
If\textup{~(\ref{d.bimods})} is a nondegenerate balanced bilinear map,
and if all of $_S A,$ $A_R,$ $_R B$ and $B_T$ are Artinian, must
these modules also be Noetherian?
\end{question}

If the answer is positive, one could look at intermediate
cases, e.g., where three of the above Artinian conditions are assumed.
(The case of~(\ref{d.otimes}) where $S$ is our given field~$k,$
$_S A = {_k k},$ and $B_T$ is an Artinian
non-Noetherian right module over a $\!k\!$-algebra $T,$
shows that the assumption that all the above modules
\emph{except} $_R B$ are Artinian is \emph{not} sufficient to prove
$B_T$ Noetherian; though it does not say whether those conditions
are sufficient to make $_S A$ and/or $A_R$ Noetherian.)

Some results with the desired sort of conclusion, but with
hypotheses of a stronger sort than those suggested
above, are proved in~\cite{GM}.

Alongside Question~\ref{Q.4artin} and its close relatives,
one might look for results with the weaker conclusion
that $A$ or $B$ have ascending chain condition on sub\emph{bimodules}.

\end{document}